\documentclass[a4paper]{amsart}

\def\eps{\varepsilon }
\def\lam{\lambda }

\def\RR{\mathbb R}

\def\NN{\mathbb N}
\def\QQ{\mathbb Q}

\def\intr{\mathop\mathrm{int}\nolimits}
\def\sgn{\mathop\mathrm{sgn}\nolimits}
\def\dis{\mathop\mathrm{dis}\nolimits}
\newcommand{\set}[1]{\{#1\}}
\providecommand{\abs}[1]{\lvert#1\rvert}

\newcommand{\remove}[1]{ }

\newtheorem{theorem}{Theorem}[section]
\newtheorem{proposition}[theorem]{Proposition}
\newtheorem{lemma}[theorem]{Lemma}
\newtheorem{corollary}[theorem]{Corollary}

\theoremstyle{definition}
\newtheorem{definition}[theorem]{Definition}

\theoremstyle{remark}
\newtheorem*{remark}{Remark}
\newtheorem*{remarks}{Remarks}
\newtheorem*{example}{Example}

\numberwithin{equation}{section}

\begin{document}
\title[A simplified integral]{A simplified multidimensional integral}
\author{\'Agnes M. Backhausz}
\address{Institute of Mathematics\\
E\"otv\"os Lor\'and University\\
P\'azm\'any P. s\'et\'any 1/C\\
1117 Budapest, Hungary
}
\email{agnes@cs.elte.hu}
\author{Vilmos Komornik}
\address{D\'epartement de math\'ematique\\
         Universit\'e Louis Pasteur\\
         7 rue Ren\'e Descartes\\
         67084 Strasbourg Cedex, France}
\email{komornik@math.u-strasbg.fr}
\author{Tivadar Szil\'agyi}
\address{Department of Applied Analysis and Computational Mathematics\\
E\"otv\"os Lor\'and University\\
P\'azm\'any P. s\'et\'any 1/C\\
1117 Budapest, Hungary}
\email{sztiv@cs.elte.hu}
\date{December 4, 2007}

\begin{abstract}
We present a simplified integral of functions of several variables. Although less general than the Riemann integral, most functions of practical interest are still integrable. On the other hand, the basic integral theorems can be obtained more quickly. We also give a characterization of the integrable functions and their primitives.
\end{abstract}

\maketitle

\section{Introduction}\label{s1}

In many undergraduate textbooks (see, e.g., \cite{Dieudonne}) the one-dimensional Riemann integral is replaced by a simpler theory: a function $f:[a,b]\to\RR$ is integrable if it is the uniform limit of a sequence of step functions, i.e., finite linear combinations of characteristic functions of bounded intervals. Although not all Riemann integrable functions are integrable in this narrower sense, many functions of practical interest are still integrable and the theory can be developed more quickly.

It is straightforward to generalize this integral to functions of several variables by changing the intervals to products of intervals but the resulting theory is less satisfactory because the class of integrable functions is too small. For example, in the two-dimensional case the characteristic functions of triangles and disks are not integrable.

A slight modification of the definition, however, leads to a substantially broader integral concept. The proofs remain short and simple but many functions of practical interest become again integrable, including the characteristic functions of Jordan measurable sets.

In Section \ref{s2} we outline briefly the general theory. The following three sections are devoted to the characterization of integrable functions, of their indefinite integrals and to the clarification of the relations between this new integral and the Riemann integral.

\section{K-integrable functions}\label{s2}

Given a positive integer $n$, by a \emph{brick} $T$ we mean a product of $n$ \emph{bounded} intervals: $T=I_1\times\cdots\times I_n$. Its \emph{volume} $\lam (T)$ is by definition the product of the lengths of the intervals. 

For the rest of this section we fix a brick $T$ and all bricks are supposed to be subsets of $T$. All functions in this section are assumed to map $T$ into $\RR$.

A \emph{step function} $g$ is by definition a finite linear combination of characteristic functions of bricks:
\begin{equation*}
g=\sum_{j=1}^M c_j\chi_{T_j}
\end{equation*}
where $M$ is a positive integer, $c_1$,\ldots, $c_M$ are real numbers and $T_1$,\ldots, $T_M$ are bricks. The \emph{integral} of a step function is defined by the usual formula:
\begin{equation*}
\int_Tg=\sum_{j=1}^M c_j\lam (T_j).
\end{equation*}
It is well-known that this integral does not depend on the particular choice of the bricks $T_j$ and that it is a positive linear form on the vector space of step functions, satisfying the estimate
\begin{equation*}
\left\vert \int_Tg\right\vert \le \max_T\abs{g}\sum_{j=1}^M \lam (T_j).
\end{equation*}

In order to extend the integral to more general functions we introduce the following convergence notion:

\begin{definition}\label{d21}
A sequence of functions $f_1$, $f_2$,\ldots \emph{converges nearly uniformly} to $f$, if
\begin{enumerate}    

\item the sequence $(f_m)$ is uniformly bounded;

\item for every $\delta>0$ there exist finitely many bricks of
total volume less than $\delta$ such that  $f_m$ converges
uniformly to $f$ on the complement of their union.
\end{enumerate}
\end{definition}

Now we generalize the integral as follows:
\begin{definition}\label{d22}
A function $f$ is \emph{K-integrable} if there is a sequence $(g_m)$ of step functions  converging nearly uniformly to $f$. Then the K-integral of $f$ is defined by the following formula:
\begin{equation*}
\int_T f := \lim_{m\to\infty} \int_T g_m.
\end{equation*}
\end{definition}

Using the above mentioned properties of the integral of step functions, one can readily verify that this limit exists, and its value does not depend on the particular choice of the sequence of step functions. (This also implies that for step functions we obtain the same integral as before.) Furthermore, one can easily establish the following proposition:

\begin{proposition}\label{p23}\mbox{}

(a) The K-integrable functions form a vector space on which the integral is a positive linear form. 
\medskip

(b) If a sequence of K-integrable functions $f_m$ converges nearly uniformly to a function $f$, then $f$ is also K-integrable, and
\begin{equation*}
\int_T f = \lim_{m\to\infty} \int_T f_m.
\end{equation*}
\end{proposition}

In order to show that many functions of practical interest are K-integrable, we recall that a set is a \emph{Jordan null set} if for each $\eps>0$ it can be covered by finitely many bricks of total volume less than $\eps$, and that a set is \emph{Jordan measurable} if it is bounded and its boundary is a Jordan null set. We say that a property holds \emph{Jordan almost everywhere} if it holds outside a Jordan null set.

We have the following results:

\begin{proposition}\label{p24}\mbox{}

(a) If $f$ is K-integrable and $g=f$ Jordan almost everywhere, then $g$ is also K-integrable and $\int_T g =\int_T f$. 
\smallskip

(b) If $f$ is continuous on a compact Jordan measurable set $C$ and vanishes on $T\setminus C$, then $f$ is K-integrable.
\smallskip

(c) If $A\subset T$ is Jordan measurable, then $\chi_A$ is K-integrable and $\int_T\chi_A$ is equal to the Jordan measure of $A$. 
\end{proposition}

\begin{proof}\mbox{}

(a) If a sequence of step functions $f_m$ converges nearly uniformly to $f$, then it converges nearly uniformly to $g$, too.
\medskip

(b) Consider a tiling of $T$ with $m^n$ translates $T_i$ of the brick $m^{-1}T$, $m=1,2,\ldots$, and introduce the step function $g_m:=\sum_{i=1}^{m^n}f(x_i)\chi_{T_i}$ where $x_i$ is the center of $T_i$. Using the uniform continuity of $f\vert_C$ and the fact that the total volume of the bricks $T_i$ meeting the boundary of $C$ tends to zero as $m\to\infty$, one can readily show that $g_m$ converges nearly uniformly to $f$.
\medskip

(c) This is a special case of (b) and $\int_T g_m$ converges to the Jordan measure of $A$ by the definition of the Jordan measure.
\end{proof}

It is easy to prove that for $n\ge 2$ the K-integral of continuous functions may be computed by successive integration:

\begin{proposition}\label{p25}
If $f:T=[a_1,b_1]\times\cdots\times [a_n,b_n]\to\RR$ is continuous $(n\ge 2)$, then
\begin{equation}\label{21}
\int f(x)\ dx
=\int_{a_1}^{b_1}\Bigl(\dots \Bigl(\int_{a_n}^{b_n}
f(x_1,\dots, x_n)\ dx_n\Bigr) \dots\Bigr) \ dx_1.
\end{equation}
\end{proposition}

\begin{proof}
Assume for simplicity that $n=2$. (The proof is easily adapted to the general case.) Using the uniform continuity of $f$ one may readily verify that the function
\begin{equation*}
x_1\mapsto \int_{a_2}^{b_2} f(x_1, x_2)\ dx_2
\end{equation*}
is well-defined and continuous on $[a_1,b_1]$, so that both sides of \eqref{21} are defined in the sense of K-integrals. 

Using the uniform continuity of $f$ again, there exists a sequence of step functions $g_m$ converging uniformly to $f$ on $[a_1,b_1]\times [a_2,b_2]$. 
Since the equality \eqref{21} clearly holds for step functions, it suffices to prove that 
\begin{equation*}
\int_{a_1}^{b_1}\Bigl(\int_{a_2}^{b_2} g_m(x_1, x_2)\ dx_2\Bigr) \ dx_1\to 
\int_{a_1}^{b_1}\Bigl(\int_{a_2}^{b_2} f(x_1, x_2)\ dx_2\Bigr) \ dx_1.
\end{equation*}
This can be proved easily: for any given $\eps>0$, we have $\abs{f-g_m}<\eps$ on $T$ for all sufficiently large $m$. For every such $m$ we then also have
\begin{equation*}
\int_{a_2}^{b_2} f(x_1, x_2)-\eps\ dx_2\le \int_{a_2}^{b_2} g_m(x_1, x_2)\ dx_2\le \int_{a_2}^{b_2} f(x_1, x_2)+\eps\ dx_2
\end{equation*}
for every $x_1\in [a_1,b_1]$. It follows that
\begin{multline*}
\int_{a_1}^{b_1}\Bigl(\int_{a_2}^{b_2} f(x_1, x_2)-\eps\ dx_2\Bigr) \ dx_1
\le \int_{a_1}^{b_1}\Bigl(\int_{a_2}^{b_2} g_m(x_1, x_2)\ dx_2\Bigr) \ dx_1\\
\le \int_{a_1}^{b_1}\Bigl(\int_{a_2}^{b_2} f(x_1, x_2)+\eps\ dx_2\Bigr) \ dx_1,
\end{multline*}
i.e.,
\begin{equation*}
\left\vert \int_{a_1}^{b_1}\Bigl(\int_{a_2}^{b_2} g_m(x_1, x_2)\ dx_2\Bigr) \ dx_1
-\int_{a_1}^{b_1}\Bigl(\int_{a_2}^{b_2} f(x_1, x_2)\ dx_2\Bigr) \ dx_1\right\vert \le \eps \cdot\lambda(T).\qedhere
\end{equation*}

\end{proof}

The nearly uniform convergence is not topological:

\begin{proposition}\label{p26}
The nearly uniform convergence of K-integrable functions cannot be derived from
a topology on the set of K-integrable functions.
\end{proposition}

\begin{proof}
We enumerate the rational vectors in $T$:
$\QQ^n\cap T=\set{q_1,q_2,\ldots,q_m,\ldots }$, and for
every $m\in \NN$ we set $g_m(q_m)=1$ and
$g_m(x)=0$ for $x\in T\setminus\set{q_m}$.
For a fixed $m\in\NN$ the constant sequence of step
functions $g_m, g_m, g_m,\ldots$ converges nearly uniformly to the constant zero
function $0$. In a topological
space this would imply that each $g_m$ is in the intersection of the neighborhoods of $0$, thus the sequence $(g_m)$ converges nearly uniformly
to the function $0$, too. But this contradicts to Definition \ref{d21}.
\end{proof}

\begin{remarks}\mbox{}

\begin{itemize}
\item The above sequence $(g_m)$ and $g\equiv 0$
show that the condition
\begin{equation*}
\lim_{m\to\infty}\inf\{\sup_{x\in T\setminus H}\left\{\left|f_m(x)-f(x)\right|\right\}\,:\, H\subset T \text{ is of Jordan measure zero}\}=0
\end{equation*}
does not imply that the sequence $(f_m)$ converges nearly uniformly to $f$.

\item On the other hand, the sequence $g_m=\chi_{\left[0,1/m\right]}$ on
$[0,1]$ converges nearly uniformly to  the constant zero function, but
\begin{equation*}
\inf\{\sup_{x\in T\setminus H}\left\{\left|g_m(x)-f(x)\right|\right\}\,:\, H\subset T \text{ is of Jordan measure zero}\}=1
\end{equation*}
for every $m\in\NN$.
\end{itemize}
\end{remarks}

It follows from the definitions that neither the K-integrability of a function $f:T\to\RR$, nor the value of the integral changes if we replace $f$ by its restriction to the \emph{open} brick $\intr (T)$. We may therefore restrict our investigation without loss of generality to functions defined on a \emph{closed} brick.

\section{Characterization of K-integrable functions: bounds and discontinuities}
\label{s3}

In this section we fix a closed brick $T=[a_1,b_1]\times\cdots\times [a_n,b_n]$ and all bricks are supposed to be subsets of $T$. All functions in this section are assumed to map $T$ into $\RR$.

We are going to describe the K-integrable functions $f$. For this we need an extension of the notion of right-hand limit and left-hand
limit to the case of functions defined on $T$ with $n>1$.

\begin{definition}\label{d31}
Given a function $f$, a point $x=(x_1,x_2,\ldots,x_n)\in \intr(T)$ and 
a nonzero vector $\alpha=(\alpha_1,\alpha_2,\ldots ,\alpha_n )\in\set{-1,0,1}^n$,
which we consider as a \emph{direction}, we set
\begin{equation*}
I_k=
\begin{cases}
[a_k,x_k)\quad &\text{if }\alpha_k=-1,\\
\set{x_k}&\text{if }\alpha_k=0,\\
(x_k,b_k]&\text{if }\alpha_k=1
\end{cases}
\end{equation*}
for $k=1,2,\ldots,n$,  and we define $T_{x, \alpha}:=I_1\times I_2\times \ldots \times I_n$.

If the restriction of $f$ to $T_{x,\alpha}$ has a limit $L$ at $x$,
then we say that $f$ \emph{has a limit at $x$ in the direction
$\alpha$}, and that this limit is equal to $L$. We write
\begin{equation*}
\lim_{x,\alpha} f=\lim\limits_{t\to x} f|_{T_{x,\alpha}}(t). 
\end{equation*}
\end{definition}

Since we assumed that $x\in \intr(T)$, $x$ is a limit point of each $T_{x,\alpha}$. Furthermore, $T\setminus \set{x}$ is the disjoint union of the sets $T_{x,\alpha}$.

There are $3^n-1$ different nonzero vectors in $\set{-1,0,1}^n\setminus\set{0}$, so we
can consider the limit of $f$ in $3^n-1$ directions. If $f$ is
continuous at $x\in \intr(T)$, then $f$ has a finite limit at $x$ in
every direction, equal to $f(x)$.

In the one-dimensional case, when $T=[a,b]$, the limits
in direction $-1$ and $1$ are the left-hand and right-hand
limits of $f$, respectively. According to this, we introduce
another definition:

\begin{definition}\label{d32}
Assume that the function $f$ is not continuous at a point $x\in\intr(T)$. 
We say that $x$ is a \emph{discontinuity point of the first kind} of $f$ if
 $f$ has a finite limit at $x$ in every direction $\alpha$. Otherwise we say that 
 $x$ is a \emph{discontinuity point of the second kind}
of $f$. We denote the set of discontinuity points of the first, 
respectively second kind of $f$ by $\dis_1(f)$ and $\dis_2(f)$.
\end{definition}

The main result of this section is the following

\begin{theorem} \label{t33}\mbox{}

(a) A function $f$ is K-integrable if and only
if the following two conditions are satisfied:

\begin{enumerate}
\item $f$ is bounded outside some Jordan null set;

\item its discontinuities of the second kind form a Jordan null set.
\end{enumerate}
\smallskip

(b) Moreover, the discontinuities of a K-integrable function may be covered by a Jordan null set $H$ and countably many hyperplanes whose normal vectors are parallel to some of the coordinate axes; in particular, a K-integrable function is continuous Lebesgue almost everywhere.
\end{theorem}

We divide the proof into several parts.

\begin{proof}[Proof of the necessity of the boundedness condition in (a)]
Let $f$ be K-integrable and consider a sequence of step
functions $g_m$ which converges nearly uniformly to $f$. By the first part of 
Definition \ref{d21} there is a
positive number $C>0$ such that $\abs{g_m(x)}<C$ for every $m\in\NN$ and $x\in T$. 
Furthermore, it follows from the second part of Definition \ref{d21} that $g_m(x)\to f(x)$ for all $x\in T$, except perhaps a Jordan null set. We conclude by observing that $\abs{f(x)}\le C$ at all these points.
\end{proof}

\begin{proof}[Proof of the necessity of the discontinuity condition in (a)]
Let $f$ be a K-in\-tegr\-able function and $\delta$ a positive number.
We show that the set $\dis_2(f)$ can be covered by finitely many closed bricks of total volume less than $\delta$. 

Since $f$ is K-integrable, by definition there is a sequence
of step functions $g_m$ which converges nearly uniformly to $f$.
There exist therefore finitely many closed bricks of total volume $<\delta$ such that $g_m$ converges uniformly to $f$ outside the union $A_{\delta}$ of these bricks. It suffices to prove the inclusion
$\dis_2(f)\subset A_{\delta}$. Equivalently, we prove that if $x\in\intr
(T)\setminus A_{\delta}$, then $f$ has a finite limit in every direction $\alpha$.

Since $A_{\delta}$ is closed,  $g_m\to f$
uniformly on an open neighborhood of $x$. Since $g_m$ is a step
function for each $m\in\NN$, it has a finite limit 
\begin{equation*}
\lim_{y,\alpha} f_m=\lim_{t\to  y} f_m|_{T_{x,\alpha}}(t)
\end{equation*}
at every point $y\in\intr(T)$,  in every direction $\alpha$. 
The uniform convergence in an open neighborhood of $x$
implies that 
\begin{equation*}
\lim_{x,\alpha} f=\lim_{t\to  x} f|_{T_{x,\alpha}}(t)
\end{equation*}
also exists and is finite in every direction. 
\end{proof}

Our proof of the sufficiency part of the theorem is based on the so-called Cousin's lemma (see for example
\cite{Bartle}, \cite{Pfeffer}):

\begin{lemma}\label{l34}
For every  positive
function $\delta:T\to (0,\infty)$ there exists a $\delta$-fine dotted partition of $T$,
i.e., a finite number of pairs $(T_1,\xi_1)$, $(T_2,\xi_2)$,\ldots, $(T_N,\xi_N)$ satisfying the following conditions:

\begin{itemize}
\item $T_i\subset T$ is a closed brick for $i=1,2,\ldots,N$;

\item $\intr (T_i )\cap\intr (T_j )=\emptyset$ if $i\ne j$;

\item $T=\bigcup_{i=1}^N T_i$;

\item $\xi_i\in T_i\subset B(\xi_i,\delta(\xi_i))$ for $i=1,2,\ldots,N$.
\end{itemize}
\end{lemma}

\begin{proof}[Proof of the sufficiency part  in (a)]
It follows from the boundedness and discontinuity conditions that for some positive $C$
\begin{equation*}
H_0=\overline{\partial T\cup\dis_2(f)\cup(T\setminus f^{-1}[-C,C])}
\end{equation*}
is a Jordan null set. There exists therefore for each $m\in \NN$ an open set $B_m$, which is the union of finitely many open
bricks of total volume $<1/m$, such that $H_0\subset B_m$.
We may also assume that $B_{m+1}\subset B_m$ for every $m\in
\NN$.

We need to define a sequence of step functions which converges nearly uniformly to
$f$. Let us fix a positive integer $m\in \NN $. First we
define a function  $p=p_m$ as follows. 
For each $x\in B_m$ we choose $p(x)>0$ such that the
open ball of center $x$ and radius $p(x)$, denoted by
$B(x,p(x))$, is a subset of $B_m$. This is possible because $B_m$ is open.
In the other case, if $x\in T\setminus B_m$, then $x\notin H_0$, hence $x$
cannot be a discontinuity of the second kind of $f$. It follows that for each direction 
$\alpha$  there is a number $p_{\alpha}(x)>0$ such that
\begin{equation*}
\left|f(y)-\lim_{x,\alpha} f\right|<1/m\quad\text{for every}\quad
y\in T_{x,\alpha}\cap B(x,p_{\alpha}(x)).
\end{equation*}
We denote by $p(x)$ the minimum of the $3^n-1$ values 
$p_{\alpha}(x)$. Then $p(x)$ is well-defined and positive for $x\in T$, and
\begin{equation*}
\left|f(y)-\lim_{x,\alpha} f\right|<1/m \quad\text{for every}\quad
x\in T\setminus B_m \text{ and } y\in T_{x,\alpha}\cap B(x,p(x)).
\end{equation*}

Applying Cousin's lemma for $\delta=p$, we obtain a $p$-fine dotted
partition of $T$. Now we define a step function $g_m$.

For any fixed point $y\in T$ we choose the smallest $i$ such that $y\in T_i$. If for
this index $i$ we have $\xi_i\in B_m$, then we set
$g_m(y)=0$. If $\xi_i\notin B_m$, then $\xi_i$ cannot be a discontinuity point of the second kind
of $f$. If $y=\xi_i$, then we set $g_m(y)=f(y)$. Finally, if 
$y\neq\xi_i$,  then there is a unique direction $\alpha$ such that $y\in T_{\xi_i,\alpha}$.
In this case we set $g_m(y)=\lim\limits_{\xi_i,\alpha}
f$. 

It is clear that for each $m\in\NN $, $g_m$ is a well-defined step function. We complete the proof of the theorem by proving that  $g_m$
converges nearly uniformly to $f$.

Since the $m$th dotted partition is $p_m$-fine, using the definition of $p_m$ and
$g_m$ we obtain for every $m\in \NN $ that
\begin{equation*}
\abs{f(y)-g_m(y)}<1/m\quad\text{for all}\quad y\in T\setminus B_m.
\end{equation*}

Furthermore, since $B_M\supseteq B_{M+1}$ for every
$M\in\NN $, this implies that $g_m\to  f$ uniformly
on $T\setminus B_M$ as $m$ approaches $\infty$. 
Since $B_M$ is the union of finitely many bricks  of total volume $<1/M$, we conclude that
the second condition of Definition \ref{d21} is satisfied.

If $g_m(y)\neq 0$ for some $m\in\NN $ and $y\in T$, then $g_m(y)$ is the value or the limit of $f$ at some
point $\xi_i\notin B_m$. Since $H_0$ is closed and $H_0\subset B_m$, then $\xi_i$ has an open neighborhood, disjoint from $H_0$,
so that $\abs{g_m(y)}\leq C$. This
means that the first condition of Definition \ref{d21} is also satisfied.
\end{proof}

\begin{proof}[Proof of part (b) of the theorem]
Let $f:T\to\RR$ be a K-integrable function and let $(g_m)$ be a sequence of step functions, nearly uniformly converging to $f$. The  discontinuities of each step function may be covered by finitely many hyperplanes of the form $x_j=c_j$ with $1\leq j\leq n$ and $c_j\in\RR$. Let us denote by $H$ the union of all these, countably many hyperplanes.

By the definition of the nearly uniform convergence, for each positive integer $M$ there exist finitely many closed bricks of total volume $<1/M$ such that $g_m$ converges uniformly to $f$ outside the union $A_M$ of these bricks. We may assume that each $A_M$ contains the boundary of $T$. The proof will be completed if we show that $f$ is continuous at every point $x\in T\setminus (H\cup A)$ with $A:=\cap A_M$.

If $x$ is such a point, then $x\in T\setminus (H\cup A_M)$ for a suitable $M$. Since $T\setminus A_M$ is an open set, $g_m$ converges uniformly to $f$ in a neighbourhood of $x$. Since, furthermore, $x\notin H$, each $g_m$ is continuous in $x$ and therefore $f$ is also continuous in $x$.
\end{proof}

Using part (b) of Theorem \ref{t33} we can prove easily the converse of part (c) of Proposition \ref{p24}:

\begin{corollary}\label{c35}
If $A\subset T$ and $\chi_A$ is K-integrable, then $A$ is Jordan measurable. 
\end{corollary}

\begin{proof}
If $\chi_A$ is K-integrable, then the set of discontinuities of $\chi_A$, which is the boundary of $A$, is a Lebesgue null set. Being compact it is also a Jordan null set, which implies that $A$ is Jordan measurable. 
\end{proof}

\section{Relation to Riemann-integrability}
\label{s4}

We fix again a closed brick $T$ and we consider only real-valued functions defined on $T$. The following proposition clarifies the relations between K-integrable and Rie\-mann-integrable functions.

\begin{proposition}\label{p41}\mbox{}

(a) If $f$ is an unbounded, K-integrable function, then there exists a \emph{bounded} K-integrable function $g$ which is equal to $f$ Jordan almost everywhere, so that $\int_Tf=\int_Tg$.
\smallskip

(b) Every bounded K-integrable function $f$ is Riemann integrable, and 
its K-integral is equal to its Riemann integral.
\smallskip

(c) There exist Riemann-integrable functions that are not K-integrable.
\end{proposition}

\begin{proof}\mbox{}

(a) We have shown at the beginning of the proof of Theorem \ref{t33} that there exists a constant $C$ such that $\abs{f}<C$ outside a Jordan null set. Then the function
\begin{equation*}
g:={\rm med\ }\set{-C,f,C}=\min\set{\max\set{-C,f},C}
\end{equation*}
is bounded and $f=g$ outside a Jordan null set. We conclude by using part (a) of Proposition \ref{p24}.
\medskip

(b) By part (b) of Theorem \ref{t33} the discontinuities of $f$ form a Lebesgue null set. Since $f$ is
bounded, this implies that $f$ is Riemann integrable.

By Definition \ref{d22} there is a sequence of step
functions $g_m$ which converges nearly uniformly to $f$.
Since $f$ is bounded, we may assume that for some $C>0$
the inequalities $\abs{f(x)}<C$ and $\abs{g_m(x)}<C$ hold
for all $x\in T$ and $m\in\NN$.

Since the $K$-integral of a step function is clearly equal to its Riemann integral, it suffices to prove that, considering all the integrals below as Riemann integrals, 
\begin{equation*}
\lim_{m\to\infty}\int_Tg_m=\int_Tf.
\end{equation*}

For a given $\eps >0$, there exist finitely many open
bricks of total volume $\lambda(A)<\eps/C$, where the union of the bricks is denoted by $A$, such that $g_m\to f$ uniformly on $T\setminus
A$. Because of the uniform convergence, for Riemann integrals we know
that
\begin{equation*}
\lim_{m\to \infty} \int_{T\setminus A} g_m=\int_{T\setminus A} f.
\end{equation*}
Hence for large values of $m$ we have
\begin{align*}
\left|\int_T g_m-\int_T f\right|
&=\left|\int_A g_m-\int_A f+\int_{T\setminus A} g_m-\int_{T\setminus A} f\right|\\
&\leq\left|\int_A g_m\right|+\left|\int_A f\right|+\left|\int_{T\setminus A} g_m-\int_{T\setminus A} f\right|\\
&\leq C\cdot\lambda(A)+C\cdot\lambda(A)+\left|\int_{T\setminus A}(g_m-f)\right|\\
&<2C\cdot\frac{\eps}{C}+\eps =3 \eps .
\end{align*}
\medskip

(c) We construct a function $f:[0,1]\to \RR $ which is Riemann-integrable but not K-integrable. We
use the construction of the fat Cantor set $S$, which is of
Lebesgue measure $1/2$.

First we remove the open middle one fourth of $[0,1]$. Then for
$k=2,3,\ldots$ we remove the open subintervals of length $4^{-k}$ from
the middle of each of the $2^{k-1}$ remaining closed intervals. Let $I_n=(a_n,b_n)$, $n=1,2,\ldots$ be an enumeration of the removed disjoint open intervals, then $S=[0,1]\setminus\bigcup_{n=1}^\infty (a_n,b_n)$ has Lebesgue measure $1/2$.

By means of this construction we define
$f:[0,1]\to \RR $. We set $f(x)=0$ if $x\in S$, and we set
\begin{equation*}
f(x)=\frac{1}{n}\cdot\left(\sin{\frac{1}{x-a_n}}+\sin{\frac{1}{b_n-x}}\right)
\end{equation*}
if $x \in I_n$ for some $n$. It is easy to see that $f$ is well-defined and the
set of its discontinuities is $\bigcup_{n=1}^\infty\set{a_n,b_n}$. This set is countable, $f$ is bounded,
hence $f$ is Riemann-integrable. But every discontinuity is of the
second kind, and the set $\bigcup_{n=1}^\infty\set{a_n,b_n}$ is not of Jordan measure zero, because if
we choose finitely many closed intervals covering this set, then they
cover $S$, which is of positive Lebesgue measure. Applying
Theorem \ref{t33} we conclude that $f$ is not K-integrable.
\end{proof}

We end this section with two examples which exhibit two unusual properties of the K-integral.

First we show Proposition \ref{p25} on the successive integration is not true for all K-integrable functions, not even for all bounded K-integrable
functions. 

\begin{example}
There exists a bounded K-integrable function
$h:[0,1]\times[0,4]\to\RR$ such that the function $[0,4]\ni y\mapsto h(x,y)$ is K-integrable
for each $x\in[0,1]$ but the function $[0,1]\ni x\mapsto \int_0^4
h(x,y) dy$ is not K-integrable.

Indeed, consider the function $f:[0,1]\to \RR $ defined in the proof of part (c) of
Proposition \ref{p41}.  We proved that $f$ is Riemann
integrable but not K-integrable. Now we set $h(x,y)=1$
if $0\leq y\leq f(x)+2$ and $h(x,y)=0$ otherwise.
It follows from the definition of $f$ that for each $m\in
\NN$, $\abs{f(x)}\leq 1/m$ except finitely many intervals, and that $f$ is continuous on the interior of
these intervals. These properties imply that  $h$ is continuous Jordan almost everywhere, and thus
$h$ is K-integrable on $[0,1]\times[0,4]$ by Theorem
\ref{t33}. It is obvious that $h$ is bounded and $[0,1]\ni x\mapsto \int_0^4
h(x,y) dy=f(x)+2$ is not K-integrable.
\end{example}

In order to investigate the rotations it is convenient to extend slightly the definition of K-integrable functions to
functions defined on arbitrary subsets of $\RR ^n$.

\begin{definition}\label{d42}
A function $f: D\to \RR $, $D\subset\RR^n$, is \emph{K-integrable} if there exists a brick $T\subset D$ such that $f|_T$ is K-integrable and $f$ vanishes in $D\setminus T$. In this case the K-integral of $f$ is defined by
\begin{equation*}
\int_D f:= \int_T f|_T.
\end{equation*}
\end{definition}

The following example shows that the K-integrability is not invariant under rotations in $\RR^2$.

\begin{example}
Let $g$ denote the rotation of center $(0,0)$ by a fixed angle $0<\alpha<\pi/2$ in the plane $\RR^2$. Consider Thomae's function $H:[0,1]\to \RR $ defined by 
$H(x):=0$ if $x$ is irrational or zero, and $H(x)=1/q$
if $x=p/q$ with relatively prime integers satisfying $p\ne 0$ and $q>0$. Then the formula $h(x,y):=H(x)$
defines a K-integrable function $h:[0,1]\times[0,1]\to \RR $ by Theorem \ref{t33} because it is bounded and has no discontinuity points of the second  kind, but $h\circ g$ is not K-integrable because the image of $(\QQ \cap(0,1])\times[0,1]$ under the rotation $g$ consists of discontinuity points of the second kind and this set is not a Jordan null set.
\end{example}

\section{Indefinite integrals of K-integrable functions}
\label{s5}

We fix again a closed brick
$T=[a_1,b_1]\times [a_2,b_2]\times\cdots\times[a_n,b_n]$ and we denote by $\Sigma$ the family of all closed bricks contained in $T$:
\begin{equation*}
\Sigma=\left\{\left[\alpha_1,\beta_1\right]\times\left[\alpha_2,\beta_2\right]\times\ldots\times
\left[\alpha_n,\beta_n\right]:a_i\leq\alpha_i\leq\beta_i\leq b_i,\ i=1,2,\ldots, n\right\}.
\end{equation*}

If $f:T\to  \RR $ a is K-integrable function, then $f|_S$ is K-integrable for all $S\in\Sigma$. Therefore the following definition is meaningful:

\begin{definition}\label{d51}
If $f:T\to  \RR $ a is K-integrable function,
then the \emph{indefinite integral} of $f$ is the set function $\Psi:\Sigma\to  \RR$ defined by the formula
\begin{equation*}
\Psi(S):=\int_S f,\quad S\in \Sigma.
\end{equation*}
\end{definition}
 
In order to characterize the indefinite integrals we need two more definitions.

\begin{definition}\label{d52}
A function $\Phi:\Sigma\to  \RR $
is \emph{strongly differentiable} at a point $u\in \intr(T)$
with strong derivative $t\in \RR $, if for every $\eps >0$
there exists $\delta>0$ such that
\begin{equation*}
\left|\frac{\Phi(S)}{\lambda(S)}-t\right|<\eps  
\end{equation*}
for every $S\in\Sigma$ satisfying $\lambda(S)>0$ and $S\subset B(u,\delta)$. The strong derivative of $\Phi$ at $u$ is denoted by $\Phi'(u)$.
\end{definition}

\begin{definition}\label{d53}
A function $\Phi:\Sigma\to  \RR $
is \emph{strongly differentiable at a point $u\in \intr(T)$
in direction} $\alpha\in \set{-1,1}^n$ with strong derivative
$t\in\RR $, if for any $\eps >0$ there exists $\delta>0$ such that
\begin{equation*}
\left|\frac{\Phi(S)}{\lambda(S)}-t\right|<\eps  
\end{equation*}
for every $S\in\Sigma$ satisfying $\lambda(S)>0$ and $S\subset B(u,\delta)\cap\overline{T_{u,\alpha}}$. The strong derivative of $\Phi$ at $u$ in direction $\alpha$ is denoted by $\Phi'_{\alpha}(u)$.
\end{definition}
We need the condition $\alpha\in\set{-1,1}^n$ because for
other directions $\alpha$ the brick $T_{u,\alpha}$ has volume zero.

The main result of this section is the following:

\begin{theorem}\label{t54}
A function $\Psi:\Sigma\to\RR$ is the indefinite integral of a suitable K-integrable function $f:T\to\RR$ if and only if $\Psi$ has the following properties:
\begin{enumerate}
\item $\Psi$ is a Lipschitz function, i.e., there exists $L>0$ such that $\abs{\Psi(S)}\leq L\cdot\lambda(S)$ for every $S\in\Sigma$;

\item $\Psi$ is finitely additive;

\item $\Psi$ is strongly differentiable Lebesgue almost everywhere;

\item $\Psi$ is strongly differentiable in every direction $\alpha\in\set{-1,1}^n$ Jordan almost everywhere.
\end{enumerate}
\end{theorem}

\begin{remark}
Omitting the last one from the four conditions we get a characterization of the indefinite integral of Riemann integrable functions.
\end{remark}

For the proof we need the following easy corollary of Definitions \ref{d32}, \ref{d52} and \ref{d53}.

\begin{lemma}\label{l55}
Let $\Psi$ be the indefinite integral of a K-integrable function $f:T\to\RR$.

If  $f$ is continuous at $u\in \intr (T)$ and $f(u)=t$,
then $\Psi$ is strongly differentiable at $u$ and $\Psi'(u)=t$.

If $f$  has a limit $t$ at $u\in \intr (T)$
in direction $\alpha\in \set{-1,1}^n$, then $\Psi$ is strongly differentiable at $u$
in direction $\alpha$ and $\Psi'_{\alpha}(u)=t$. 
\end{lemma}

\begin{proof}[Proof of the necessity part of Theorem \ref{t54}]
If $f:T\to\RR$ is K-integrable, then by
Theorem \ref{t33} there exists $C>0$ such
that $\set{x\in T:\abs{f(x)}>C}$ is a Jordan null set.
Hence $\left|\int_S
f\right|<C\cdot\lambda (S)$ for every $S\in\Sigma$, because it is well
known for Riemann integrals, and parts (a) and (b) of Proposition \ref{p41} show that this implies the inequality for
K-integrals.
We conclude that $\Psi$ satisfies the
Lipschitz condition with the constant $C$:
\begin{equation*}
\abs{\Psi(S)}\leq C\cdot\lambda(S)\quad\text{for every}\quad S\in\Sigma.
\end{equation*}

Similarly, we deduce from the relation between the K-integral and the Riemann integral 
that $\Psi$ is finitely additive:
\begin{equation*}
\Psi(S)=\sum_{i=1}^m\Psi(S_i)
\end{equation*}
if $S\in\Sigma$, $S_i\in \Sigma$ for $i=1,2,\ldots ,m$, $S=\bigcup_{i=1}^m S_i$ and
$\intr (S_i)\cap \intr (S_j)=\emptyset$ whenever $i\ne j$.
By part (b) of Theorem \ref{t33} the set of discontinuities of $f$
is a Lebesgue null set, and by Lemma \ref{l55} this
implies that $\set{u\in T: \Psi'(u) \text{ does not exist}}$
is also a Lebesgue null set. Similarly, by Theorem \ref{t33}
the set of discontinuities of the second kind of $f$ is a Jordan null set,
and by Lemma \ref{l55} this implies that
\begin{equation*}
\set{u\in T: \exists \alpha\in\set{-1,1}^n\text{ such that $\Psi'_{\alpha}(u)$ does not exist}}
\end{equation*}
is a Jordan null set, too.
\end{proof}

The rest of this section is devoted to the proof of the sufficiency part of Theorem \ref{t54}.
Given a function $\Psi:\Sigma\to\RR$ satisfying the four conditions of the theorem, we define a function $f:T\to\RR$ by setting 
$f(x)=0$ for $x\in \partial T$ and
\begin{equation*}
f(x)=\lim_{\delta\to 0+} \sup \left\{\frac{\Psi(S)}{\lambda(S)}:S\in\Sigma,  \lambda(S)>0, S\subset B(x,\delta)\right\}
\quad\text{for}\quad x\in \intr (T).
\end{equation*}
We are going to prove in several steps that $f$ is K-integrable and that $\Psi$ is its indefinite integral.

\begin{lemma}\label{l56}
If the strong derivative $\Psi'_{\alpha}(u)=t$ exists at a point $u\in \intr(T)$  in some direction $\alpha\in\set{-1,1}^n$, then  $\lim_{u,\alpha} f =t$.
\end{lemma}

\begin{proof}
Given $\eps >0$ arbitrarily, by definition there exists $\delta_0>0$ such that 
\begin{equation*}
t-\eps <\frac{\Psi(S)}{\lambda(S)}<t+\eps 
\end{equation*}
for all bricks $S\in\Sigma$ satisfying $\lambda(S)>0$ and $S\subset B(u, \delta_0)\cap\overline{T_{u,\alpha}}$. We may assume that $B(u, \delta_0)\subset T$. It suffices to show that $\abs{f(x)-t}\leq\eps $ for every $x\in B(u,\delta_0)\cap T_{u,\alpha}$.

Since $\alpha\in\set{-1,1}^n$, $T_{u,\alpha}\cap B(u,\delta_0)$ is an open set, for every $x\in B(u,\delta_0)\cap T_{u,\alpha}$ there exists $\delta_1>0$ such that $B(x,\delta_1)\cap T\subset B(u,\delta_0)\cap T_{u,\alpha}$. This implies that for every positive number $0<\delta<\delta_1$ we have
\begin{equation*}
t-\eps \leq\sup \left\{\frac{\Psi(S)}{\lambda(S)}:S\in\Sigma,  \lambda(S)>0,  S\subset B(x,\delta)\right\}\leq t+\eps , 
\end{equation*}
It follows that
\begin{equation*}
t-\eps \leq \lim_{\delta\to 0+}\sup \left\{\frac{\Psi(S)}{\lambda(S)}:S\in\Sigma, \ \lambda(S)>0,\  S\subset B\left(x,\delta\right)\right\}\leq t+\eps ,
\end{equation*}
i.e., $\abs{f(x)-t}\leq\eps $. 
\end{proof}

\begin{lemma}\label{l57}
If the strong derivative $\Psi'_{\alpha}(u)=t$ exists at a point $u\in \intr(T)$  in \emph{every} direction $\alpha\in\set{-1,1}^n$, then  
$u$ cannot be a discontinuity point of the second kind of $f$.
\end{lemma}

\begin{proof}
By the previous lemma, it is sufficient to prove that $f$ has a finite limit $\lim_{u,\alpha}f$ in each nonzero direction $\alpha\in \set{-1,0,1}^n\setminus\set{-1,1}^n$. We claim that this limit is equal to $t:=\max\set{\Psi'_{\beta}(u)\ :\ \beta\in B}$ where 
\begin{equation*}
B=\set{\beta\in\set{-1,1}^n\ : \text{ if $1\le k\le n$ and $\alpha_k\neq0$, then $\beta_k=\alpha_k$}}.
\end{equation*}

Let $\eps >0$ be given. According to Definition \ref{d53} there exists $\delta_0>0$ such that
\begin{equation*}
\left|\frac{\Psi(\sigma )}{\lambda(\sigma )}-\Psi'_{\beta}(u)\right|<\eps 
\end{equation*}
for every $\beta\in B$ and $\sigma\in\Sigma$ satisfying $\lambda(\sigma )>0$ and $\sigma\subset B(u,\delta_0)\cap \overline{T_{u,\beta}}$.

If $x\in T_{u, \alpha}\cap B(u,\delta_0)$, then there exists $0<\delta_x<\delta_0$ such that $B(x,\delta_x)\subset B(u,\delta_0)$ and
\begin{equation*}
\delta_x<\min\set{\abs{x_k-u_k}\ :\ x_k\neq u_k}.
\end{equation*}
Let us consider a brick $S\in\Sigma$ such that $\lambda(S)>0$ and $S\subset B(x,\delta_x)\subset B(u,\delta_0)$. 

The definition of $\delta_x$ implies that $\sgn(y_k-u_k)=\sgn(x_k-u_k)$ for every $y\in S$ and for each positive integer $k\le n$ satisfying $x_k\neq u_k$. Therefore $S$ is covered by the pairwise nonoverlapping bricks $\overline{T_{u,\beta}}$ ($\beta\in B$).
Using the finite additivity of $\Psi$ and setting
$S_{\beta}=S\cap\overline{T_{u,\beta}}$ for $\beta\in B$ we get
\begin{equation*}
\Psi(S)=\sum_{\beta\in B} \Psi(S_{\beta}).
\end{equation*}
Since the volume function $\lambda:\Sigma\to\RR $ is finitely additive, too, we have also
\begin{equation*}
\lambda(S)=\sum_{\beta\in B} \lambda(S_{\beta})
\end{equation*}
for the same sets $S$ and $S_{\beta}$. 

It follows from our choice of $\delta_0$ that
\begin{equation*}
\Psi(S_{\beta})< (\Psi'_{\beta}(u)+\eps )\cdot \lambda(S_{\beta})\le (t+\eps )\lambda(S_{\beta}). 
\end{equation*}
Summing them for all $\beta\in B$ we obtain that
\begin{equation*}
\Psi(S)=\sum_{\beta\in B} \Psi(S_{\beta})<(t+\eps )\sum_{\beta\in B}\lambda(S_{\beta})=(t+\eps )\lambda(S).
\end{equation*}
Since this holds for every $S\in\Sigma$ satisfying $\lambda(S)>0$ and $S\subset B(x,\delta_x)$, it follows that
\begin{equation*}
f(x)=\lim_{\delta\to 0+} \sup \left\{\frac{\Psi(S)}{\lambda(S)}:S\in \Sigma, \lambda(S)>0, S\subset B(x,\delta)\right\}\leq t+\eps .
\end{equation*}

In order to obtain a lower bound of $f(x)$  we choose a direction $\beta_{\rm max}\in B$ such that $\Psi'_{\beta_{\rm max}}(u)=t$. Since $x\in T_{u, \alpha}\cap B(u,\delta_0)$ and $\beta_{\rm max}\in B\subset\set{-1,1}^n$, for every $\delta>0$ there is a brick $S_{\delta}\in \Sigma$ such that $S_{\delta}\subset B(x,\delta)\cap T_{u,\beta_{\rm max}}$. If $\delta<\delta_x<\delta_0$, then $S_{\delta}\subset B(x,\delta_x)\subset 
B(u,\delta_0)$. Thanks to the choice of $\delta_0$, we conclude that     
\begin{equation*}
\frac{\Psi\left(S_{\delta}\right)}{\lambda (S_\delta )}>t-\eps .
\end{equation*}
Since there exists such a brick $S_{\delta}$ for every $0<\delta<\delta_x$, this implies that
\begin{equation*}
f(x)=\lim_{\delta\to 0+} \sup \left\{\frac{\Psi(S)}{\lambda(S)}\ :\ S\in \Sigma, \lambda(S)>0,  S\subset B(x,\delta)\right\}\geq t-\eps .
\end{equation*}

We have proved that for every $\eps >0$ there exists $\delta_0>0$ such that
\begin{equation*}
t-\eps \leq f(x)\leq t+\eps 
\end{equation*}
for every $x\in B(u,\delta_0)\cap T_{u,\alpha}$. In other words, $\lim_{u,\alpha}f=t$.
\end{proof}

Now we are ready to prove that $f$ is K-integrable. By Theorem \ref{t33} it suffices to show that $f$ is bounded and that its discontinuities of the second kind form a Jordan null set. The boundedness of $f$ follows from the Lipschitz property of $\Psi$. The second property follows from Lemma \ref{l57} and from the fourth condition on $\Psi$.

According to Hypothesis (3) and Lemma \ref{l56} $\Psi$ is strongly differentiable and $\Psi'(u)=f(u)$ Lebesgue almost everywhere. Since $f$ is K-integrable, it follows from the already established necessity part of Theorem \ref{t54} that its indefinite integral $\Psi_0$ is also strongly differentiable and $\Psi_0'(u)=f(u)$ Lebesgue almost everywhere. It remains to prove that $\Psi=\Psi_0$. This follows from the following lemma applied for $\Phi=\Psi-\Psi_0$.

\begin{lemma}\label{l58}
Assume that $\Phi:\Sigma\to\RR $ is finitely additive and satisfies the Lipschitz condition. If $\Phi$ is strongly differentiable and $\Phi'(u)=0$ Lebesgue almost everywhere, then $\Phi(S)=0$ for every $S\in\Sigma$.
\end{lemma}

\begin{proof}
We adapt a method of M. W. Botsko \cite{Botsko}. We fix a Lebesgue null set $H\subset \intr(T)$ such that $\Phi'(u)=0$ for every $u\in \intr(T)\setminus H$ and we choose  $L>0$ such that $\abs{\Phi(\sigma )}\leq L\cdot \abs{\lambda(\sigma )}$ for every $\sigma\in \Sigma$. 

Assume on the contrary that $\abs{\Phi(S)}=c>0$ for some brick $S\in\Sigma$. We may assume that $S\subset \intr(T)$. Since $H$ is a Lebesgue null set, there exist open bricks $S_1,S_2,\ldots$ such that
\begin{equation*}
H\subset \bigcup_{i=1}^{\infty} S_i\quad\text{and}\quad \sum_{i=1}^{\infty} \lambda (S_i)<\frac{c}{2L}.
\end{equation*}

We define a function $p:S\to (0,\infty)$ as follows. If $x\in H\cap S$, then we choose the first brick $S_i$ which contains $x$, and we define $p(x)$ such that $B(x,p(x))\subset S_i$. If $x\in S\setminus H$, then $\Phi$ is strongly differentiable at $x$ with strong derivative $0$. According to Definition \ref{d52} we may fix $p(x)>0$ such that
\begin{equation*}
\left|\frac{\Phi(\sigma )}{\lambda(\sigma )}\right|<\frac{c}{2 \lambda(S)}\quad\text{for every $\sigma\in\Sigma$ satisfying $\sigma\subset B(x,p(x))$ and $\lambda(\sigma )>0$}.
\end{equation*}

Applying Cousin's lemma \ref{l34} there is a $p$-fine dotted partition of $S$:
\begin{multline*}
S=\bigcup_{j=1}^J T_j\quad\text{with}\quad \xi_j\in T_j\subset B(\xi_j,p(\xi_j)),\quad j=1,\ldots,J,\\
\text{and}\quad \intr(T_j)\cap \intr(T_k)=\emptyset\quad\text{whenever}\quad j\ne k.
\end{multline*}

Since the bricks $T_j$ are pairwise nonoverlapping and $\Phi$ is additive, we get
\begin{equation*}
c=\abs{\Phi(S)}=\left|\sum_{j=1}^J \Phi(T_j)\right|\leq \sum_{\xi_j\in H} \abs{\Phi(T_j)}+\sum_{\xi_j\notin H} \abs{\Phi(T_j)}.
\end{equation*}
In view of the definition of $L$, $p$ and $S_i$ we obtain that
\begin{equation*}
\sum_{\xi_j\in H} \abs{\Phi(T_j)}\leq L\cdot\sum_{\xi_j\in H} \lambda(T_j)\leq L\cdot \sum_{i=1}^{\infty} \lambda(S_i)<L\cdot\frac{c}{2L}=\frac{c}{2}.
\end{equation*}
On the other hand, using the definition of $p$ for $x\notin H$ we get
\begin{equation*}
\sum_{\xi_j\notin H} \abs{\Phi(T_j)}<\sum_{\xi_j\notin H} \lambda(T_j)\cdot \frac{c}{2\lambda(S)}\leq \frac{c}{2}
\end{equation*}
because the union of bricks $T_j$ is $S$.

It follows from these inequalities that
\begin{equation*}
c\leq \sum_{\xi_j\in H} \abs{\Phi(T_j)}+\sum_{\xi_j\notin H} \abs{\Phi(T_j)}<\frac{c}{2}+\frac{c}{2}=c,
\end{equation*}
a contradiction.
\end{proof}

\end{document}